\documentclass[11pt]{article}
\usepackage{amscd, amssymb, amsmath, amsthm, amsfonts}
\usepackage{latexsym, graphics, graphicx,psfrag}
\usepackage[all]{xy}

\newtheorem{thm}{Theorem}[section]

\newtheorem{lem}[thm]{Lemma}
\newtheorem{prop}[thm]{Proposition}
\theoremstyle{definition}
\newtheorem{defn}[thm]{Definition}
\theoremstyle{remark}

\theoremstyle{example}

\numberwithin{equation}{section}

\begin{document}
\date{}
\title{Self-mapping Degrees of 3-Manifolds
\footnotetext{\textbf{Keywords:} 3-manifolds, self-mapping degree}
\footnotetext{\textbf{Subject class:} 57M99, 55M25}}

\author{Hongbin Sun, Shicheng Wang, Jianchun Wu, Hao Zheng  }

\maketitle

\begin{abstract} For each closed oriented $3$-manifold $M$ in
Thurston's picture, the set of degrees of self-maps on $M$ is given.
\\\

\end{abstract}
\tableofcontents
\section{Introduction}

\subsection {Background}

Each closed oriented $n$-manifold $M$ is naturally associated with a
set of integers, the degrees of all self-maps on $M$, denoted as
$D(M)=\{deg(f)\ |\  f :M\to M\}$.

Indeed the calculation of $D(M)$ is a classical topic appeared in
many literatures. The result is simple and well-known for dimension
$n=1,2$.  For dimension $n>3$, there are many interesting special
results (see \cite{DW} and references therein), but it is difficult
to get general results, since there are no classification results
for manifolds of dimension $n>3$.

The case of dimension 3 becomes attractive in the topic and it is
possible to calculate $D(M)$ for any closed oriented 3-manifold $M$.
Since Thurston's geometrization conjecture, which seems to be
confirmed, implies that closed oriented 3-manifolds can be
classified in a reasonable sense.

Thurston's geometrization conjecture claims that the each
Jaco-Shalen-Johanson decomposition piece of a prime 3-manifold
supports one of the eight geometries, which are $H^3$,
$\widetilde{PSL}(2,R)$, $H^2\times E^1$, Sol, Nil, $E^3$, $S^3$ and
$S^2\times E^1$ (for details see \cite{Th} and \cite{Sc}). Call a
closed orientable 3-manifold $M$ is {\it geometrizable} if each
prime factor of $M$ meets Thurston's geometrization conjecture. All
3-manifolds discussed in this paper are geometrizable.

The following result is known in early 1990's:

{\bf Theorem 1.0} {\it Suppose $M$ is a geometrizable 3-manifold.
Then $M$ admits a self-map of degree larger than 1 if and only if
$M$ is either

(a) covered by a torus bundle over the circle, or

(b) covered by $F\times S^1$ for some compact surface $F$ with
$\chi(F)< 0$, or

(c) each prime factor of $M$ is covered by $S^3$ or $S^2\times
E^1$.}

Hence for any 3-manifold $M$ not listed in (a)-(c) of Theorem 1.0,
$D(M)$ is either $\{0,1,-1\}$ or $\{0,1\}$, which depends on whether
$M$ admits a self map of degree $-1$ or not. To determine $D(M)$ for
geometrizable 3-manifolds listed in (a)-(c) of Theorem 1.0, let's
have a close look of them.

For short,  we often call a 3-manifold supporting Nil geometry {\it
a Nil 3-manifold, and so on}. Among Thurston's eight geometries, six
of them belong to the list (a)-(c) in Theorem 1.0.
 3-manifolds in (a) are exactly
those supporting either $E^3$, or Sol or Nil geometries.  $E^3$
3-manifolds, Sol 3-manifolds, and some Nil 3-manifolds are torus
bundle or semi-bundles; Nil 3-manifolds which are not torus bundles
or semi-bundles are Seifert spaces having Euclidean orbifolds with
three singular points. 3-manifolds in (b) are exactly those
supporting  $H^2\times E^1$ geometry; 3-manifolds supporting $S^3$
or $S^2\times E^1$ geometries form a proper subset of (3). Now we
divide all 3-manifolds in the list (a)-(c) in Theorem 1.0 into the
following five classes:

Class 1. $M$ supporting either $S^3$ or $S^2\times E^1$ geometries;

Class 2. each prime factor of $M$ supporting either $S^3$ or
$S^2\times E^1$ geometries, but $M$ is  not in Class 1;

Class 3. torus bundles and torus semi-bundles;

Class 4. Nil 3-manifolds not in Class 3;

Class 5. $M$ supporting $H^2\times E^1$ geometry.

$D(M)$ is known recently for $M$ in Class 1 and Class 3. We will
calculate $D(M)$ for $M$ in the remaining three classes. For the
convenience of the readers, we will present $D(M)$ for $M$ in all
those five classes. To do this, we need first to coordinate
3-manifolds in each class, then state the results of $D(M)$ in term
of those coordinates. This is carried in the next section.

\subsection{Main Results}

{\bf Class 1.} According to \cite{Or} or \cite{Sc}, the fundamental
group of a 3-manifold supporting $S^3$-geometry is among the
following eight types: $\mathbb{Z}_p$\,, $D_{4n}^*$\,, $T_{24}^*$\,,
$O_{48}^*$\,, $I_{120}^*$\,,
$T_{8\cdot\,3^{^q}}^{'}$\,,$D_{n'\cdot\,2^{^q}}'$ and
$\mathbb{Z}_m\times \pi_1(N)$, where $N$ is a $S^3$ 3-manifold,
$\pi_1(N)$ belongs to the previous seven ones, and $|\pi_1(N)|$ is
coprime to $m$. The cyclic group $Z_p$ is realized by lens space
$L(p,q)$, each group in the remaining types is realized by a unique
$S^3$ 3-manifold. Note also the sub-indices of those seven types
groups are exactly their orders, and the order of the groups in the
last type is $m|\pi_1(N)|$. There are only two closed orientable
$3$-manifolds supporting $S^2\times \mathbb{E}^1$ geometry:
$S^2\times S^1$ and $RP^3\#RP^3$ .

\begin{thm}\label{1}
(1) $D(M)$ for $M$ supporting $S^3$-geometry are listed below:

\begin{table}[h]
\begin{center}
\renewcommand{\arraystretch}{1.3}
\begin{tabular}{|c|c|}
  \hline
  $\pi_1(M)$ & $D(M)$\\
  \hline
  $\mathbb{Z}_p$ & $\{\,k^2\,|\,k\in\mathbb{Z}\}+p\,\mathbb{Z}$\\
  \hline
  $D_{4n}^*$ &
  $\{\,h^2\,|\,h\in\mathbb{Z};\,2\nmid\,h\;\text{or}\;h=n\,\text{or}\;h=0\}+4n\mathbb{Z}$\\
  \hline
  $T_{24}^*$ & $\{\,0,1,16\}+24\mathbb{Z}$\\
  \hline
  $O_{48}^*$ & $\{\,0,1,25\}+48\mathbb{Z}$\\
  \hline
  $I_{120}^*$ & $\{\,0,1,49\}+120\mathbb{Z}$\\
  \hline
  $T_{8\cdot3^q}^{'}$ & $\left\{\begin{array}{ll}
      \{\,k^2\cdot\,(3^{2q-2p}-3^{q})\;\;|\,3\nmid\,k, q\geq\,p>0\}+8\cdot3^q\mathbb{Z}&(\,2\,|\,q\,)\\
      \{\,k^2\cdot\,(3^{2q-2p}-3^{q+1})\;\;|\,3\nmid\,k,q\geq\,p>0\}+8\cdot3^q\mathbb{Z}&(\,2\nmid\,q\,)
      \end{array}\right.$\\
  \hline
  $D_{n'\cdot\,2^q}'$ &
    $\begin{array}{ll}&\{\,k^2\cdot[\,1-(n')^{2^{^{\,q}}-1}\,]^i\cdot
      [\,1-2^{\,(\,2p-q\,)(\,n'-1\,)}\,])^j\;|\,i,j,k,p\in\mathbb{Z},\\
      &\qquad\qquad\qquad\,q\geq\,p>0\}+n'\,2^{\,q}\,\mathbb{Z}\end{array}$\\
  \hline
  $\mathbb{Z}_m\times \pi_1(N)$ & $\left\{d\in\mathbb{Z}\,\left| \begin{array}{l}
    d=h+|\pi_1(N)|Z,\;h\in\,D(N)\\
    d=k^2+mZ,\,k\in\mathbb{Z}\end{array}\right.\right\}$
    \\
  \hline
\end{tabular}
\end{center}
\end{table}

(2) $D(S^2\times S^1)=D(RP^3\#RP^3)=\mathbb{Z}.$
\end{thm}

{\bf Class 2.} We assume that each $3$-manifold $P$ supporting
$S^3$-geometry has the canonical orientation induced from the
canonical orientation on $S^3$. When we change the orientation of
$P$, the new oriented $3$-manifold is denoted by $\bar{P}$.
Moreover, lens space $L(p,q)$ is orientation reversed homeomorphic
to $L(p,p-q)$, so we can write all the lens spaces connected
summands as $L(p,q)$. Now we can decompose each 3-manifold in Class
2 as
$$M=(mS^2\times
S^1)\#(m_1P_1\#n_1\bar{P_1})\#\cdots\#(m_sP_s\#n_s\bar{P_s})$$
$$\#(L(p_1,q_{1,1})\#\cdots\#L(p_1,q_{1,r_1}))\#\cdots\#(L(p_t,q_{t,1})\#\cdots\#L(p_t,q_{t,r_t})),$$
where all the $P_i$ are $3$-manifolds with finite fundamental group
different from lens spaces, all the $P_i$ are different with each
other, and all the positive integer $p_i$ are different from each
other. Define
$$D_{iso}(M)=\{deg(f)\ |\ f:\ M\rightarrow M,\ f\ {\rm induces\
an \ isomorphism \ on \ }\pi_1(M)\}.$$

\begin{thm} \label{2}

(1) $D(M)=D_{iso}(m_1P_1\#n_1\bar{P_1})\bigcap \cdots \bigcap
D_{iso}(m_sP_s\#n_s\bar{P_s}) \bigcap$

$D_{iso}(L(p_1,q_{1,1})\#\cdots\#L(p_1,q_{1,r_1}))\bigcap \cdots
\bigcap D_{iso}(L(p_t,q_{t,1})\#\cdots\#L(p_t,q_{t,r_t}));$

(2) $D_{iso}(mP\#n\bar{P})=\left\{ \begin{array}{cc}
         D_{iso}(P) & if \ m\ne n,\\
         D_{iso}(P)\bigcup(-D_{iso}(P))& if \ m=n;
         \end{array} \right.$

(3) $D_{iso}(L(p,q_1)\#\cdots\#L(p,q_n))= H^{-1}(C)$.

\end{thm}

The notions $H$ and $C$ in Theorem \ref{2} (3) is defined as below:

Let $U_p=\{{\rm all\ units\ in\ ring}\ \mathbb{Z}_p\}$,
$U^2_p=\{a^2\ |\ a\in U_p\}$, which is a subgroup of $U_p$.  We
consider the quotient $U_p/U_p^2=\{a_1,\cdots,a_m\}$, every $a_i$
corresponds with a coset $A_i$ of $U_p^2$. For the structure of
$U_p$, see \cite{IR} page 44. Define $H$ to be the natural
projection from $\{n\in \mathbb{Z}\ |\ gcd(n,p)=1\}$ to $U_p/U_p^2$.

Define $\bar{A}_s=\{L(p,q_i)\ |\ q_i\in A_s\}$ (with repetition
allowed). In $U_p/U_p^2$, define $B_l=\{a_s\ |\ \# \bar{A}_s=l\}$
for $l=1,2,\cdots$, there are only finitely many $l$ such that
$B_l\ne \emptyset$. Let $C_l=\{a\in U_p/U_p^2\ |\ a_ia\in B_l,\
\forall a_i \in B_l\}$ if $B_l\ne \emptyset$ and $C_l=U_p/U_p^2$
otherwise. Define $C=\bigcap_{l=1}^{\infty} C_l$.

{\bf Class 3.} To simplify notions, for a diffeomorphism $\phi$ on
torus $T$, we also use $\phi$ to present its isotopic class and its
induced 2 by 2 matrix on $\pi_1(T)$ for a given basis.

A {\it torus bundle} is $M_\phi = T \times I /(x,1) \sim
(\phi(x),0)$ where $\phi$ is a diffeomorphism of the torus $T$ and
$I$ is the interval $[0,1]$. Then the coordinates of $M_\phi$ is
given as below:

 (1) $M_\phi$ admits $E^3$ geometry, $\phi$ conjugates to a matric
of finite order $n$, where $n\in \{1,  2, 3,  4,   6\}$;

(2) $M_\phi$ admits Nil geometry, $\phi$ conjugates to $\pm\left(
                           \begin{array}{cc}
                             1 & n \\
                             0 & 1 \\
                           \end{array}
                         \right)$, where $n\neq 0$;

(3) $M_\phi$ admits Sol geometry, $\phi$ conjugates to $\left(
                           \begin{array}{cc}
                             a & b \\
                             c & d \\
                           \end{array}
                         \right)$, where $|a+d|>2, ad-bc=1$.

A {\it torus semi-bundle} $N_\phi=N\bigcup_\phi N$ is obtained by
gluing two copies of $N$ along  their  torus boundary $\partial N$
via a diffeomorphism $\phi$, where $N$ is the twisted $I$-bundle
over the Klein bottle. We have the double covering $p: S^1\times
S^1\times I\to N=S^1\times S^1\times I/\tau$, where $\tau$ is an
involution such that $\tau(x,y,z)=(x+\pi,-y,1-z)$.

Denote by $l_0$ and $l_\infty$ on $\partial N$ be the images of the
second $S^1$ factor and first $S^1$ factor on $S^1\times S^1\times
\{1\}$. {\it A canonical coordinate} is an orientation of $l_0$ and
$l_\infty$, hence there are four choices of canonical coordinate on
$\partial N$. Once canonical coordinates on each $\partial N$ are
chosen, $\phi$  is identified with an element
$\left(\begin{array}{cc}
                                                    a & b \\
                                                    c & d \\
                                                  \end{array}
                                                \right)$ of $GL_2(\mathbb{Z})$ given by
$\phi$ $(l_0,l_\infty)$= $(l_0,l_\infty)$ $\left(\begin{array}{cc}
                                                    a & b \\
                                                    c & d \\
                                                  \end{array}
                                                \right)$.

With suitable choice of canonical coordinates of $\partial N$,
$N_\phi$ has coordinates as below:

(1) $N_\phi$ admits $E^3$ geometry, $\phi=\left(
                                                  \begin{array}{cc}
                                                    1 & 0 \\
                                                    0 & 1 \\
                                                  \end{array}
                                                \right)$ or
                                                 $\left(
                                                  \begin{array}{cc}
                                                    0 & 1 \\
                                                    1 & 0 \\
                                                  \end{array}
                                                \right)$;

(2) $N_\phi$ admits Nil geometry, $\phi=\left(
\begin{array}{cc}
1 & 0 \\
z & 1 \\
\end{array}
\right)$, $\left(
\begin{array}{cc}
0 & 1 \\
1 & z \\
\end{array}
\right)$ or $\left(
\begin{array}{cc}
1 & z \\
0 & 1 \\
\end{array}
\right)$, where $z\neq 0$;

(3) $N_\phi$ admits Sol geometry, $\phi=\left(
                           \begin{array}{cc}
                             a & b \\
                             c & d \\
                           \end{array}
                         \right)$, where $abcd\neq 0, ad-bc=1$.

\begin{thm}\label{3}
 $D(M_\phi)$ is in the table below for torus bundle $M_\phi$,
  where $\delta (3)=\delta (6)=1, \delta (4)=0$.
\begin{table}[!h]
\begin{center}
\begin{tabular}{|c|c|c|}
  \hline
  $M_\phi$ & $\phi$ & $D(M_\phi)$\\
  \hline
  $E^3$ & finite order $k=1,2$ & $\mathbb{Z}$\\
  \hline
  $E^3$  & finite order $k=3,4,6$ & $\{(kt+1)(p^2-\delta (k) pq+q^2)|\ t,p,q\in\mathbb{Z}\}$ \\
  \hline
  Nil  &  $\pm\left(\begin{array}{cc}1 & 0 \\n & 1 \\ \end{array}\right), n\neq 0$ & $\{l^2|\ l\in\mathbb{Z}\}$ \\
  \hline
  Sol &  $\left(\begin{array}{cc}a & b \\c & d \\ \end{array}\right), |a+d|>2$ & $\{p^2+\frac{(d-a)pr}{c}-\frac{br^2}{c}|\ p,r\in\mathbb{Z},$ \\ & & either $\frac{br}{c}, \frac{(d-a)r}{c} \in
\mathbb{Z}$ or
 $\frac{p(d-a)-br}{c} \in \mathbb{Z}\}$ \\
  \hline
\end{tabular}
\end{center}
\end{table}

(2) $D(N_\phi)$ is listed in the table below for torus semi-bundle
$N_\phi$, where $\delta(a,d)= \frac{ad}{gcd(a,d)^2}$.
\begin{table}[!h]
\begin{center}
\begin{tabular}{|c|c|c|}
  \hline
  $N_\phi$ & $\phi$& $D(N_\phi)$\\
   \hline
  $E^3$ & $\left(\begin{array}{cc}1 & 0 \\0 & 1 \\ \end{array}\right)$ & $\mathbb{Z}$ \\
  \hline
  $E^3$ & $\left(\begin{array}{cc}0 & 1 \\1 & 0 \\ \end{array}\right)$ & $\{2l+1|\ l\in\mathbb{Z}\}$ \\
  \hline
  Nil & $\left(\begin{array}{cc}1 & 0 \\z & 1 \\ \end{array}\right), z\neq 0$ & $\{l^2|\ l\in\mathbb{Z}\}$ \\
  \hline
  Nil & $\left(\begin{array}{cc}0 & 1 \\1 & z \\
  \end{array}\right)$ or $\left(\begin{array}{cc}1 & z \\0 & 1 \\ \end{array}\right), z\neq 0$ & $\{(2l+1)^2|\ l\in\mathbb{Z}\}$ \\
  \hline
  Sol & $\left(\begin{array}{cc}a & b \\c & d \\ \end{array}\right), abcd\neq 0, ad-bc=1$&
  $\{(2l+1)^2|\ l\in\mathbb{Z}\}$, if $\delta(a,d)$ is even\ or\\ & &
       $\{(2l+1)^2|\ l\in\mathbb{Z}\}\bigcup\{(2l+1)^2\cdot \delta(a,d)$\\ & &
$|\ l\in\mathbb{Z}\}$, if $\delta(a,d)$ is odd \\
  \hline
\end{tabular}
\end{center}
\end{table}
\end{thm}

To coordinate 3-manifolds in Class 4 and Class 5, we first recall
the well known coordinates of Seifert manifolds.

 Suppose an oriented
3-manifold $M'$ is a circle bundle with a given section $F$, where
$F$ is a compact surface with boundary components $c_1,...,c_n$ with
$n>0$. On each boundary component of $M'$, orient $c_i$ and the
circle fiber $h_i$ so that the product of their orientations match
with the induced orientation of $M'$ (call such pairs $\{(c_i,
h_i)\}$ a section-fiber coordinate system). Now attach $n$ solid
tori $S_i$ to the $n$ boundary tori of $M'$ such that the meridian
of $S_i$ is identified with slope $r_i=c_i^{\alpha_i}h_i^{\beta_i}$
where $\alpha_i>0, (\alpha_i,\beta_i)=1$. Denote the resulting
manifold by $M(\pm
g;\frac{\beta_1}{\alpha_1},\cdots,\frac{\beta_s}{\alpha_s})$ which
has the Seifert fiber structure extended from the circle bundle
structure of $M'$,  where $g$ is the genus of the section $F$ of
$M$, with the sign $+$ if $F$ is orientable and $-$ if $F$ is
nonorientable, here 'genus' of nonorientable surfaces means the
number of $RP^2$ connected summands. Call $e(M)=\sum_1^s
\frac{\beta_i}{\alpha_i}\in\mathbb{Q}$ the Euler number of the
Seifert fiberation.

{\bf Class 4.} If a Nil manifold $M$ is not a torus bundle  or torus
semi-bundle, then $M$ has one of the following Seifert fibreing
structures:
$M(0;\frac{\beta_1}{2},\frac{\beta_2}{3},\frac{\beta_3}{6})$,
$M(0;\frac{\beta_1}{3},\frac{\beta_2}{3},\frac{\beta_3}{3})$,
 or $M(0;\frac{\beta_1}{2},\frac{\beta_2}{4},\frac{\beta_3}{4})$, where $e(M) \in \mathbb{Q}-\{ 0\}$.

\begin{thm}\label{4} For 3-manifold $M$ in Class 4, we have

(1)
$D(M(0;\frac{\beta_1}{2},\frac{\beta_2}{3},\frac{\beta_3}{6}))=\{l^2|l=m^2+mn+n^2,
l\equiv 1 \mod 6, m,n \in \mathbb{Z}\}$;

(2)
$D(M(0;\frac{\beta_1}{3},\frac{\beta_2}{3},\frac{\beta_3}{3}))=\{l^2|l=m^2+mn+n^2,
l\equiv 1 \mod 3, m,n \in \mathbb{Z}\}$;

(3)
$D(M(0;\frac{\beta_1}{2},\frac{\beta_2}{4},\frac{\beta_3}{4}))=\{l^2|l=m^2+n^2,
l\equiv 1 \mod 4, m,n \in \mathbb{Z}\}$.
\end{thm}

{\bf Class 5.} All manifolds supporting $H^2\times E^1$ geometry are
Seifert manifolds $M$ such that $e(M)=0$ and the Euler
characteristic of the orbifold $\chi(O_M)<0$.

Suppose
$M=(g;\frac{\beta_{1,1}}{\alpha_1},\cdots,\frac{\beta_{1,m_1}}{\alpha_1},\cdots,
\frac{\beta_{n,1}}{\alpha_n},\cdots,\frac{\beta_{n,m_n}}{\alpha_n})$,
where all the integers $\alpha_i>1$ are different from each other,
and $\sum_{i=1}^n \sum_{j=1}^{m_i} \frac{\beta_{i,j}}{\alpha_i}=0$.

For each $\alpha_i$ and each $a\in U_{\alpha_i}$, define
$\theta_a(\alpha_i)=\#\{\beta_{i,j}\ |\ p_i(\beta_{i,j})=a\}$ (with
repetition allowed), $p_i$ is the natural projection from $\{n\ |\
gcd(n,\alpha_i)=1\}$ to $U_{\alpha_i}$. Define $B_l(\alpha_i)=\{a\
|\ \theta_a(\alpha_i)=l\}$ for $l=1,2,\cdots$, there are only
finitely many $l$ such that $B_l(\alpha_i)\ne \emptyset$. Let
$C_l(\alpha_i)=\{b\in U_{\alpha_i}\ |\ ab\in B_l(\alpha_i),\ \forall
a \in B_l(\alpha_i)\}$ if $B_l(\alpha_i)\ne \emptyset$ and
$C_l(\alpha_i)=U_{\alpha_i}$ otherwise. Finally define
$C(\alpha_i)=\bigcap_{l=1}^{\infty} C_l(\alpha_i)$, and
$\bar{C}(\alpha_i)=p_i^{-1}(C(\alpha_i))$.

\begin{thm}\label{5} $D(M(g;\frac{\beta_{1,1}}{\alpha_1},\cdots,\frac{\beta_{1,m_1}}{\alpha_1},\cdots,\frac{\beta_{n,1}}{\alpha_n},\cdots,\frac{\beta_{n,m_n}}{\alpha_n}))=\bigcap_{i=1}^n
\bar{C}(\alpha_i).$
\end{thm}

\subsection{A brief comment of the topic and organization of the paper}

Theorem 1.0 was appeared in \cite{Wa}. The proof of the "only if"
part in Theorem 1.0 is based on the results on simplicial volume
developed by Gromov, Thurston and Soma (see \cite{So}), and various
classical results by others on 3-manifold topology and group theory
(\cite{He}, \cite{SW}, \cite{R}). The proof of "if" part in Theorem
1.0 is a sequence elementary constructions, which were known before,
for example see \cite{HL} for (3). That graph manifolds admits no
self-maps of degrees $>1$ also follows from  a recent work
\cite{De}.

The table in Theorem \ref{1} is quoted from \cite{Du}, which
generalizes the earlier work \cite{HKWZ}. The statement below quoted
from \cite{HKWZ}  will be repeatedly used  in this paper.

\begin{prop}\label{6}
For $3$-manifold $M$ supporting $S^3$ geometry,
$$D_{iso}(M)=\{k^2+l|\pi_1(M)|\ |\ gcd(k,|\pi_1(M)|)=1\}.$$
\end{prop}

 The topic of mapping degrees between (and to) 3-manifolds covered by $S^3$
 has been discussed for long times.  We mention several
papers: in very old papers \cite {Rh} and \cite{Ol}, the degrees of
maps between any give pairs of lens spaces are obtained by using
equivalent maps between spheres; in \cite{HWZ},  $D(M, L(p,q))$ can
be computed for any 3-manifold $M$; and in a very recent one
\cite{MP}, an algorithm (or formula) is given to the degrees of maps
between given pairs of 3-manifolds covered by $S^3$ in term of their
Seifert invariants.

Theorem \ref{3} is proved in \cite{SWW}.

Theorems \ref{2}, \ref{4} and \ref{5} will be proved in Sections 3,
4, and 5 respectively in this paper. In Section 2 we will compute
$D(M)$ for some concrete 3-manifolds using Theorems 1-5. We will
also discuss when $-1\in D(M)$ and when $-1\in D(M)$ implies that
$M$ admits orientation reversing homeomorphisms.

All terminologies not defined are standard, see \cite{He}, \cite{Sc}
and \cite{IR}.

\section{Examples of computation, orientation reversing homeomorphisms}

{\bf Example 2.1} Let $M=(P\#\bar{P})\#(L(7,1)\#L(7,2)\#2L(7,3))$,
where $P$ is the Poincare homology three sphere.

By Theorem \ref{2} (2), Proposition \ref{6} and the fact
$|\pi_1(P)|=120$,  we have $D(P\#\bar{P})=D_{iso}(P)\cup
(-D_{iso}(P))=\{120n+i\ |\ n\in \mathbb{Z},\ i=1,49, 71,119\}$.

Now we are going to calculate $D((L(7,1)\#L(7,2)\#2L(7,3))$
following the notions of Theorem \ref{2} (3). Clearly
$U_7=\{1,2,3,4,5,6\}$ and $U_7^2=\{1,2,4\}$. Then $U_7/U_7^2=\{a_1,
a_2\}$, where $a_1=\bar 1$ and $a_2=\bar 3$; $U_7=\{A_1\cup A_2\}$,
where $A_1=U_7^2, A_2=3U_7^2$; $\#\bar A_1=2$ and $\#\bar A_2=2$;
$B_2=\{\bar 1, \bar 3\}$, $B_l=\emptyset$ for $l\ne 2$. Since
$U_7/U_7^2=B_2$, we have $C_2=B_2$ and also $C_l=U_7/U_7^2$ for
$l\ne 2$; then  $C=\bigcap_{l=1}^{\infty} C_l=U_7/U_7^2$. Then for
the natural projection $H: \{n\in \mathbb{Z}\ |\ gcd(n,p)=1\}\to
U_7/U_7^2$, $H^{-1}(C)$ are all number coprime to $7$, hence we have
$D_{iso}((L(7,1)\#L(7,2)\#2L(7,3))=\{l\in \mathbb{Z}\ |\
gcd(l,7)=1\}$ by Theorem \ref{2} (3).

Finally by Theorem \ref{2} (1), we have $D(M)=\{120n+i\ |\ n\in
\mathbb{Z},\ i=1,49,71,119\}\bigcap \{l\in \mathbb{Z}\ |\
gcd(l,7)=1\}= \{840n+i\ |\ n\in\mathbb{Z},\ i=1,71, 121,169,191,
239, 241,289, 311, 359, 361,409, $ $431, 479, 481,529, 551, 599,
601,649, 671, 719, 769, 839.\}$. Note $-1\in D(M)$.

\vskip 0.3 true cm

{\bf Example 2.2} Suppose
$M=(2P\#\bar{P})\#(L(7,1)\#L(7,2)\#L(7,3))$.

Similarly by Theorem \ref{2} (2), Proposition \ref{6} and
$|\pi_1(P)|=120$, we have $D(2P\#\bar{P})=D_{iso}(P)=\{120n+i\ |\
n\in \mathbb{Z},\ i=1,49\}$.

To calculate $D(L(7,1)\#L(7,2)\#2L(7,3))$, we have $U_7$, $U_7^2$,
$U_7/U_7^2=\{a_1, a_2\}$, $U_7=\{A_1, A_2\}$ exactly as last
example. But then $\#\bar A_1=2$ and $\#\bar A_2=1$; $B_1=\{ \bar
3\}$, $B_2=\{ \bar 1\}$, $B_l=\emptyset$ for $l\ne 1, 2$. Moreover
$C_1=C_2=\{ \bar 1\}$, and $C_l=U_7/U_7^2$ for $l\ne 1, 2$; then
$C=\bigcap_{l=1}^{\infty} C_l=\{ \bar 1\}$, and $H^{-1}(C)=\{7n+i\
|\ n\in \mathbb{Z},\ i=1,2,4\}$. Hence we have
$D_{iso}(\#(L(7,1)\#L(7,2)\#L(7,3))=\{7n+i\ |\ n\in \mathbb{Z},\
i=1,2,4\}$ by Theorem \ref{2} (3).

By Theorem \ref{2} (1), $D(M)=\{120n+i\ |\ n\in \mathbb{Z},\
i=1,49\}\bigcap \{7n+i\ |\ n\in \mathbb{Z},\ i=1,2,4\}=\{840n+i\ |\
n\in\mathbb{Z},\ i=1,121,169,289,361,529.\}$ Note $-1\notin D(M)$.

\vskip 0.3 true cm

{\bf Eaxmple 2.3} By Theorem  \ref{3}, for the torus bundle
$M_\phi$, $\phi=\left(
                           \begin{array}{cc}
                             2 & 1 \\
                             1 & 1 \\
                           \end{array}\right)$, among the first 20
                            integers $>0$, exactly $1, 4, 5, 9,
                           11 , 16, 19, 20\in D(M_\phi)$.
\vskip 0.3 true cm

{\bf Example 2.4} For Nil 3-manifold
$M=M(0;\frac{\beta_1}{2},\frac{\beta_2}{3},\frac{\beta_3}{6})$,
$D(M)=\{l^2|l=m^2+mn+n^2, l\equiv 1 \mod 6, m,n \in \mathbb{Z}\}$.
The numbers in $D(M)$ smaller than $10000$ are exactly
$1,49,169,361,625,961,1369,1849,2401,3721, 4489, 5329, 6241, 8291,
9409$. Since all $l=6k+1,\ k\in\mathbb{N}$ with $l^2\le 10000$ can
be presented as $m^2+mn+n^2$ except  $l=55, 85$ (if $5|m^2+mn+n^2$,
then $5|(2m+n)^2+3n^2$, therefore $5|2m+1$ and $5|n$, it follows
that $25|m^2+mn+n^2$).

\vskip 0.3 true cm

{\bf Example 2.5} For $\mathbb{H}^2\times \mathbb{E}^1$ manifold
$M=M(2;\frac{1}{5},\frac{1}{5},-\frac{2}{5},\frac{1}{7},\frac{2}{7},-\frac{3}{7})$,
we follow the notions in Theorem \ref{5} to calculate $D(M)$.

First we have $U_5=\{1,2,3,4\}$ with indices $\theta_a(5)$ are
$\{2,0,1,0\}$ respectively. Then $B_1(5)=\{3\}$, $B_2(5)=\{1\}$,
$B_l(5)=\emptyset$ for $l\ne 1,2$ and $C_1(5)=C_2(5)=\{1\}$. Hence
$C(5)=\bigcap_{l=1}^{\infty} C_l(5)=\{  1\}$. Hence $\bar
C(5)=\{5n+1\ |\ n\in \mathbb{Z}\}.$

Similarly $U_7=\{1,2,3,4,5,6\}$ with indices $\theta_a(7)$ are
$\{1,1,0,1,0,0\}$ respectively. Then $B_1(7)=C_1(7)=\{1,2,4\}$.
$B_l(7)=\emptyset$ and $C_l(7)=U_7$  for $l\ne 1$. Hence
$C(7)=\bigcap_{l=1}^{\infty} C_l(7)=\{ 1,2,4\}$.$\bar C(7)=\{7n+i\
|\ n\in\mathbb{Z},\ i=1,2,4\}$.

Finally $D(M)=\{5n+1\ |\ n\in \mathbb{Z}\} \bigcap \{7n+i\ |\
n\in\mathbb{Z},\ i=1,2,4\}=\{35n+i\ |\ n\in\mathbb{Z},\
i=1,11,16\}.$

\vskip 0.3 true cm

{\bf Example 2.6} Suppose $M$ is a 3-manifold supporting $S^3$
geometry. By Proposition \ref{6}, $M$ admits degree $-1$ self
mapping if and only if there is integer number $h$, such that
$h^2\equiv -1\ mod\ \pi_1(M)$. Then we can prove that if $M$ is not
a lens space, $-1\notin D(M)$, (see proof of Proposition \ref{m=n}).
With some further topological and number theoritical arguments, the
follow results were proved in \cite{Sun}.

(1) There is a degree $-1$ self map on $L(p,q)$, but no orientation
reversing homeomorphism on it if and only if $(p,q)$ satisfies:
$p\nmid q^2+1$, $4\nmid p$ and all the odd prime factors of $p$ are
the $4k+1$ type.

(2) Every degree $-1$ self map on $L(p,q)$ are homotopic to an
orientation reversing homeomorphism if and only if $(p,q)$
satisfies: $q^2 \equiv -1\ \text{mod}\ p$,
$p=2,p_1^{e_1},2p_1^{e_1}$, where $p_1$ is a $4k+1$ type prime
number. \vskip 0.3 true cm

{\bf Example 2.7} Suppose $M$ is a torus bundle. Then any non-zero
degree map is homotopic to a covering (\cite{Wa} Cor 0.4). Hence if
$-1\in D(M)$, then $M$ admits an orientation reversing self
homeomorphism.

(1) For the torus bundle $M_\phi$, $\phi=\left(
                           \begin{array}{cc}
                             2 & 1 \\
                             1 & 1 \\
                           \end{array}\right)$,  $-1 \in D(M_\phi)$.
Indeed for $\phi=\left(\begin{array}{cc}
                             a & b \\
                             c & d \\
                           \end{array}\right)$, if  $|a+d|=3$,
                           then $-1 \in
                           D(M_\phi)$.
Since $p^2+\frac{d-a}{b} pr-\frac{c}{b} r^2=-1$ has solution
$p=1-d$, $r=b$
             when $a+d=3$, and solution $p=-1-d$, $r=b$ when $a+d=-3$.

(2) For the torus bundle $M_\phi$, $\phi=\left(
                           \begin{array}{cc}
                             2 & 3 \\
                             1 & 2 \\
                           \end{array}\right)$,  $-1 \notin D(M_\phi)$.
Indeed for $\phi=\left(\begin{array}{cc}
                             a & b \\
                             c & d \\
                           \end{array}\right)$, if  $a+d\pm 2$ has prime decomposition $p_1^{e_1}...p_n^{e_n}$ such that
                           $p_i=4l+3$ and $e_i=2m+1$ for some $i$,
                            then $-1 \notin
                           D(M_\phi)$. Since if the equation $p^2+\frac{d-a}{b} pr-\frac{c}{b}
                           r^2=-1$ has integer solution, then
                           $\frac{((a+d)^2-4)r^2-4b^2}{b^2}$ should
                           be a square of rational number. That is $((a+d)^2-4)r^2-4b^2=s^2$ for some integer
                           $s$.
           Therefore $(a+d+2)(a+d-2)r^2$ is a sum of two squares. By
           a fact in elementary number theory,
            neither $a+d+2$ nor $a+d-2$ has $4k+3$ type prime factor with odd power (see  page 279, \cite{IR}.)

\section{$D(M)$ for connected sums}
\subsection{Relations between $D_{iso}(M_1\# M_2)$ and $\{D_{iso}(M_1),
D_{iso}(M_2)\}$}

In this section, we consider the manifolds $M$ in Class 2: $M$ has
non trivial prime decomposition, each connected summand has finite
or infinite cyclic fundamental group, and $M$ is not homeomorphic to
$RP^3\#RP^3$. (Note for each geometrizable 3-manifold $P$,
$\pi_1(P)$ is finite if and only if $P$ is $S^3$ 3-manifold, and
$\pi_1(P)$ is infinite cyclic if and only if $P$ is $S^2\times E^1$
3-manifold.)

Since each $S^3$ 3-manifold $P$  is covered by $S^3$, we assume $P$
has the canonical orientation induced by the canonical orientation
on $S^3$.  When we change the orientation of $P$, the new oriented
$3$-manifold is denoted by $\bar{P}$. Moreover, lens space $L(p,q)$
is orientation reversed homeomorphic to $L(p,p-q)$, so we can write
all the lens spaces connected summands as $L(p,q)$. Now we can
decompose the manifold as
$$M=(mS^2\times
S^1)\#(m_1P_1\#n_1\bar{P_1})\#\cdots\#(m_sP_s\#n_s\bar{P_s})$$
$$\#(L(p_1,q_{1,1})\#\cdots\#L(p_1,q_{1,r_1}))\#\cdots\#(L(p_t,q_{t,1})\#\cdots\#L(p_t,q_{t,r_t})),$$
where all the $P_i$ are $3$-manifolds with finite fundamental group
different from lens spaces, all the $P_i$ are different with each
other, and all the positive integer $p_i$ are different from each
other. We will use this convention in this section.

Suppose $F$ (resp. $P$) is a properly embedded surface (resp. an
embedded 3-manifold) in a 3-manifold $M$.  We use $M\setminus F$
(resp. $M\setminus P$) to denote the resulting manifold obtained by
splitting $M$ along $F$ (resp. removing $\text{int} P$, the interior
of $P$).

The definitions below are quoted from \cite{R}:

\begin{defn}
Let $M, N$ be 3-manifolds and $B_f=\bigcup_i(B_i^+\cup B_i^-)$ is a
finite collection of disjoint 3-ball pairs in $int M$. A map $f:
M\setminus B_f\rightarrow N$ is called an {\it almost defined map}
from $M$ to $N$ if for each $i$, $f|_{\partial B_i^+}=f|_{\partial
B_i^-}\circ r_i$ for some orientation reversing homeomorphism $r_i:
\partial B_i^+\rightarrow
\partial B_i^-$. If identifying $\partial B_i^+$ with $\partial
B_i^-$ via $r_i$, we get a quotient closed manifold $M(f)$, and $f$
induces a map $\tilde{f}: M(f)\rightarrow N$. We define
$deg(f)=deg(\tilde{f})$.
\end{defn}

\begin{defn}
For two almost defined maps $f$ and $g$, we say that $f$ is {\it
$B$-equivalent} to $g$ if there are almost defined maps $f=f_0,
f_1,\cdots, f_n=g$ such that either $f_i$ is homotopic to $f_{i+1}$
 rel$(\partial B_{f_i}\cup \partial B_{f_{i+1}})$ or $f_i=f_{i+1}$ on
$M\setminus B$ for an union of balls $B$ containing $B_{f_i}\cup
B_{f_{i+1}}$.
\end{defn}

\begin{lem}[\cite{R} Lemma 3.6, \cite{Wa} Lemma 1.11]\label{almost}
Suppose $f: M\rightarrow M$ is a map of nonzero degree and $\bigcup
S_i^2$ is an union of essential 2-spheres. Then there is an almost
defined map $g: M\setminus B_g\rightarrow M$, $B$-equivalent to $f$,
such that $deg(g)=deg(f)$ and $g^{-1}(\bigcup S_i^2)$ is a
collection of spheres.
\end{lem}

\begin{lem}[\cite{Wa} Corollary 0.2]\label{iso} Suppose $M$ is a
geometrizable  3-manifold. Then any nonzero degree proper map $f:
M\rightarrow M$ induces an isomorphism $f_*: \pi_1(M)\rightarrow
\pi_1(M)$ unless $M$ is covered by either a torus bundle over the
circle, or $F\times S^1$ for some compact surface $F$, or the $S^3$.
\end{lem}

The following lemma is well-known.

\begin{lem}\label{suj}
Suppose $M$ is a closed orientable 3-manifold, $f: M\rightarrow M$
is of degree $d\neq 0$. Then $f_*: H_2(M,\mathbb{Q})\rightarrow
H_2(M,\mathbb{Q})$ is an isomorphism.
\end{lem}

\begin{thm}\label{perm}
Suppose $M=M_1\#\cdots\#M_n$ is a non-prime manifold which is not
homeomorphic to $RP^3\#RP^3$. Each $\pi_1(M_i)$ is finite or cyclic,
and $\pi_1(M_i)\ne 0$. If $f: M\rightarrow M$ is a map of degree
$d\neq 0$, then there exists a permutation $\tau$ of
$\{1,\cdots,n\}$, such that there is a map $g_i:
M_{\tau(i)}\rightarrow M_i$ of degree $d$ for each $i$. Moreover,
$g_{i*}$ is an isomorphism on fundamental group.
\end{thm}

\begin{proof} Call $M'$ is a punctured $M$, if $M'=M\setminus B$,
where $B$ is a finitely many  disjoint 3-balls in the interior of
$M$. We use $\widehat{M_*}$ to denote the 3-manifold obtained from
$M_*$ by capping off the boundary spheres with 3-balls.

$M$ is obtained by gluing the boundary sphere of $M'_i=M_i\setminus
int(B_i)$ to a $n$-punctured 3-sphere. The image of $\partial B_i$
in $M$, which is denoted by $S_i$, is a separating sphere.

By Lemma \ref{almost}, there is an almost defined map $g: M\setminus
 B_g\rightarrow M$, $B$-equivalent to $f$, such that $g^{-1}(\bigcup
S_i)$ is a collection of spheres and $deg(g)=d$. Let $M_g=M\setminus
 B_g$.

Let $U=M_g\setminus g^{-1}(\bigcup S_i)=\{M_i^{j}| j=1,...,l_i, \,
i=1,...,n\}$. The components of $g^{-1}(M'_i)$ are denoted by
$M^1_i,\cdots,M^{l_i}_i$.

By Lemma \ref{iso}, $f_*: \pi_1(M)\rightarrow \pi_1(M)$ is an
isomorphism. Since $g$ is differ from $f$ just on the 3-balls $B_g$
up to homotopy rel $\partial B_g$, it follows that $g_*:
\pi_1(M\setminus B_g)=\pi_1(M)\rightarrow \pi_1(M)$ is an
isomorphism.

Since the prime decomposition of 3-manifold $M$ is unique, and $M_g$
is just a punctured $M$, each component of $U$ is either a punctured
non-trivial prime factor of $M$, or a punctured 3-sphere.

By Lemma \ref{suj}, $f_*$ is an injection on $H_2(M,\mathbb{Q})$. If
$S_i$ is a separating sphere, then $[S_i]=0$ in $H_2(M,\mathbb{Q})$.
So each component $S'$ of $f^{-1}(S_i)$ is homologic to 0, thus $S'$
separates $M$. By the procession of construction of $g$ (see the
proof of Lemma 3.4, \cite{R}), which is B-equivalent to $f$, each
component $S$ of $g^{-1}(S_i)$ is also a separating sphere in $M_g$.
So $\pi_1(M_g)$ is the free product of the $\pi_1(M_i^j)$,
$i=1,...,n, j=1,...,l_i$.

Note $\pi_1(M)=\pi_1(M_1)*...*\pi_1(M_n)$, each $\pi_1(M_i)$ is an
indecomposable factor of $\pi_1(M)$. Since $g_*$ is an isomorphism
and each punctured 3-sphere has trivial $\pi_1$, from the basic fact
on free product of groups, it follows that there is at least one
punctured prime non-trivial factor in $g^{-1}(M'_i)$. Since this is
true for each $i=1,...n$ and there are at most $n$ punctured prime
non-trivial factors in $U$, it follows that there are $n$ punctured
prime non-trivial factors in $U$. Hence there is exactly one
punctured prime non-trivial factor in $g^{-1}(M'_i)$, denoted as
$M_{\tau(i)}$, moreover $g_*: \pi_1(M_{\tau(i)})\rightarrow
\pi_1(M_i)$ is an isomorphism, where $\tau$ is a permutation on
$\{1,...,n\}$.

Since $\pi_1(M_i)=Z$ if and only if $M_i=S^2\times S^1$, it follows
that if $M_i=S^2\times S^1$, then $M_{\tau(i)}=S^2\times S^1$. Since
$D(S^2\times S^1)=Z$, below we assume that $\widehat M_i'\ne
S^2\times S^1$, and to show that there is a map $g_i: M_{\tau(i)}\to
M_i$ of degree $d$.

Since the map $g: M(g)\to M$ has degree $d$, then $g_i=g|:
(\cup_{j=1}^{l_i} M_i^j)(g)\to M'_i$ is a proper map of degree $d$,
which can extend to a map $\widehat{g_i}: \widehat{(\cup_{j=1}^{l_i}
M_i^j)(g)}\to \widehat{M'_i}=M_i$ of degree $d$ between closed
3-manifolds. The last map is also defined on
$\widehat{(\cup_{j=1}^{l_i} M_i^j)(g)}\setminus \overline {\partial
B_g} =(\cup_{j=1}^{l_i} \widehat{M_i^j)}\setminus B_g\subset
(\cup_{j=1}^{l_i} \widehat{M_i^j)}$, where $\overline {\partial
B_g}\subset M(g)$ is the image of $ {\partial B_g}\subset M$.

Now consider the map $\widehat{g_i}:(\cup_{j=1}^{l_i}
\widehat{M_i^j)}\setminus B_g \to M_i$. Since $\pi_2(M_i)=0$, we can
extend the map $\widehat{g_i}$ from $\cup_{j=1}^{l_i}
\widehat{M_i^j}\setminus B_g$ to $\cup_{j=1}^{l_i} \widehat{M_i^j}$.
More carefully, for each pair $B_k^+, B_k^- \subset \cup_{j=1}^{l_i}
\widehat{M_i^j}$ we can make the extension with the property
$\widehat{g_i}|_{B_i^+}=\widehat{g_i}|_{B_i^-}\circ \widehat{r_i}$,
where $\widehat {r_i}: B_i^+\rightarrow B_i^-$ is an orientation
reversing homeomorphism extending $r_i:
\partial B_i^+\rightarrow
\partial B_i^-$.
Now it is easy to see  the map $\widehat g_i: \cup_{j=1}^{l_i}
\widehat{M_i^j}\to M_i$ is still of degree $d$.

From the map $\widehat g_i: (\cup_{j=1}^{l_i} \widehat{M_i^j)}\to
M_i$ one can obviously define a map $ g_i: \#_{j=1}^{l_i}
\widehat{M_i^j}\to M_i$ of degree $d$ between connected 3-manifolds.
Since all $\widehat{M_i^j}$ are $S^3$ except one is $M_{\tau(i)}$,
we have map $g_i: M_{\tau(i)}\to M_i$.
\end{proof}

\begin{defn} For  closed oriented 3-manifold $M$, $M'$, define
$$D_{iso}(M, M')=\{deg(f)\ |\ f:\ M\rightarrow M',\ f\ {\rm induces\
isomorphism\ on \ fundamental\ group}\},$$
$$D_{iso}(M)=\{deg(f)\ |\ f:\ M\rightarrow M,\ f\ {\rm induces\
isomorphism\ on \ fundamental\ group}\}.$$
\end{defn}

 Under the condition we considered in this
section,  we have $D(M)=D_{iso}(M)$ by Lemma \ref{iso}.

\begin{lem} \label{supset} Suppose $f_i: M_i\to M_i'$ is a map of
degree $d$ between closed $n$-manifolds, $n\ge 3$, $f_{i*}$ is
surjective on $\pi_1$, $i=1,2$. Then there is a map $f: M_1\# M_2\to
M_1'\# M_2'$ of degree $d$ and $f_{*}$ is surjective on $\pi_1$. In
particular,

(1) $D_{iso}(M_1\# M_2, M'_1\# M'_2)\supset D_{iso}(M_1, M'_1)\cap
D_{iso}(M_2, M'_2),$

(2) $D_{iso}(M_1\# M_2)\supset D_{iso}(M_1)\cap D_{iso}(M_2).$
\end{lem}

\begin{proof}
Since $f_{*}$ is surjective on $\pi_1$, it is known (see [RW] for
example), we can homotopy $f_i$ such that for some $n$-ball $D'_i
\subset M'_i$, $f_i^{-1}(D_i)$ is an $n$-ball ${D}_i \subset M_i$.
Thus we get a proper map $\bar{f_i}:\ M_i\setminus D_i \rightarrow
M'_i\setminus D'_i$ of degree $d$, which also induces a degree $d$
map from $\partial D_i$ to $\partial D'_i$. Since maps of the same
degree between $(n-1)$-spheres are homotopic, so after proper
homotopy, we can paste $\bar{f}_1$ and $\bar{f}_2$ along the
boundary to get map $f: M_1\# M_2\to M_1'\# M_2'$ of degree $d$ and
$f_{*}$ is surjective on $\pi_1$.
\end{proof}

\subsection{$D(M)$ for connected sums}

Suppose
$$M=(mS^2\times
S^1)\#(m_1P_1\#n_1\bar{P_1})\#\cdots\#(m_sP_s\#n_s\bar{P_s})$$
$$\#(L(p_1,q_{1,1})\#\cdots\#L(p_1,q_{1,r_1}))\#\cdots\#(L(p_t,q_{t,1})\#\cdots\#L(p_t,q_{t,r_t})),$$
where all the $P_i$ are $3$-manifolds with finite fundamental group
different from lens spaces, all the $P_i$ are different with each
other, and all the positive integer $p_i$ are different from each
other.

To prove Theorem \ref{2}, we need only to prove the three
propositions below.

\begin{prop} \label{=}$$D(M)=D_{iso}(m_1P_1\#n_1\bar{P_1})\bigcap \cdots \bigcap
D_{iso}(m_sP_s\#n_s\bar{P_s})\bigcap
D_{iso}(L(p_1,q_{1,1})\#\cdots\#L(p_1,q_{1,r_1}))$$ $$\bigcap \cdots
\bigcap D_{iso}(L(p_t,q_{t,1})\#\cdots\#L(p_t,q_{t,r_t})).  \qquad
(\ast)$$
\end{prop}

\begin{proof} For every self-mapping degree $d$ of $M$, in Theorem \ref{perm}
we have proved that for every oriented connected summand $P$ of $M$,
it corresponds to an oriented connected summand $P'$, such that
there is a degree $d$ mapping $f:\ P\rightarrow P'$, and $f$ induces
isomorphism on fundamental group. By the classification of
$3$-manifolds with finite fundamental group (see \cite{Or}, 6.2),
$P$ and $P'$ are homeomorphism (not considering the orientation)
unless they are lens spaces with same fundamental group. Now by
Lemma \ref{supset} (1), we have $d\in D_{iso}(m_iP_i\#n_i\bar{P_i})$
and $d\in D_{iso}(L(p_j,q_{j,1})\#\cdots\#L(p_j,q_{j,r_j}))$, for
$i=1,...,s$ and $j=1,...,t$. Hence we have proved
$$D(M)\subset
D_{iso}(m_1P_1\#n_1\bar{P_1})\bigcap \cdots \bigcap
D_{iso}(m_sP_s\#n_s\bar{P_s})\bigcap
D_{iso}(L(p_1,q_{1,1})\#\cdots\#L(p_1,q_{1,r_1}))$$ $$\bigcap \cdots
\bigcap D_{iso}(L(p_t,q_{t,1})\#\cdots\#L(p_t,q_{t,r_t})).$$ (Since
$D(mS^2\times S^1)=\mathbb{Z}$, we can just forget it in the
discussion.)

Apply Lemma \ref{supset} once more, we finish the proof.
\end{proof}

\begin{prop}\label{m=n} If $P$ is a $3$-manifold with finite fundamental group different from lens space,
$D_{iso}(mP\#n\bar{P})=\left\{ \begin{array}{cc}
         D_{iso}(P) & if \ m\ne n,\\
         D_{iso}(P)\bigcup(-D_{iso}(P))& if \ m=n.
         \end{array} \right.$
\end{prop}

\begin{proof}
If $P$ is not a lens space, from the list in \cite{Or}, we can check
that $4| |\pi_1(P)|$. By Proposition \ref{6},
$D_{iso}(Q)=\{k^2+l|\pi_1(Q)|\ |\ gcd(k,|\pi_1(Q)|)=1\}$, where $Q$
is any $3$-manifolds with $S^3$ geometry. If
$k^2+l|\pi_1(P)|=-k'^2-l'|\pi_1(P)|$, then
$k^2+k'^2=-(l+l')|\pi_1(P)|$. Since $4| |\pi_1(P)|$ and
$gcd(k,|\pi_1(P)|)=gcd(k',|\pi_1(P)|)=1$, $k,k'$ are both odd, thus
$-(l+l')|\pi_1(P)|=k^2+k'^2=4s+2$, contradicts with $4| |\pi_1(P)|$.
So $D_{iso}(P)\bigcap (-D_{iso}(P))=\emptyset$. (In particular
$-1\ne D(P)$.)

From the definition we have $D_{iso}(P)=D_{iso}(\bar P)$ and
$D_{iso}(P,\bar P)=D_{iso}(\bar P, P)=-D_{iso}(\bar P)$.

 If
$m\ne n$, we may assume that $m>n$. For the self-map $f$, if some
$P$ corresponds to $\bar{P}$, there must also be some $P$
corresponds to $P$, so $deg(f)\in D_{iso}(P)\bigcap (-D_{iso}(P))$,
it is impossible by the argument in first papragraph. So all the $P$
correspond to $P$, and all the $\bar{P}$ correspond to $\bar{P}$.
Since $D_{iso}(P)=D_{iso}(\bar P)$, we have  $D_{iso}(mP\#
n\bar{P})\subset D_{iso}(P)$. By Lemma \ref{supset} and the fact
$D_{iso}(P)=D_{iso}(\bar P)$, we have $D_{iso}(mP\# n\bar{P})=
D_{iso}(P)$.

If $m=n$, similarly we have either all the $P$ correspond to $P$ and
all the $\bar{P}$ correspond to $\bar{P}$; or all the $P$ correspond
to $\bar{P}$ and all the $\bar{P}$ correspond to $P$. Since
$D_{iso}(P)=D_{iso}(\bar P)$ and $D_{iso}(P,\bar P)=D_{iso}(\bar P,
P)=-D_{iso}(\bar P)$, we have  $D_{iso}(mP\# m\bar{P})\subset
D_{iso}(P)\bigcup (-D_{iso}(P))$. On the other hand from the
argument above, we have  $D_{iso}(P), -D_{iso}(P)\subset
D_{iso}(mP\# m\bar{P})$,  hence $D_{iso}(mP\# m\bar{P})=
D_{iso}(P)\bigcup (-D_{iso}(P))$.
\end{proof}

\begin{lem}\label{lens}
$D_{iso}(L(p,q),L(p,q'))=\{k^2q^{-1}q'+lp\ |\ gcd(k,p)=1\}$, here
$q^{-1}$ is seen as in group $U_p=\{{\rm all\ the\ units\ in\ the \
ring}\ \mathbb{Z}_p\}$.
\end{lem}

\begin{proof}
$L(p,q)$ is the quotient of $S^3$ by the action of $\mathbb{Z}_p$,
$(z_1,z_2)\rightarrow
(e^{i\frac{2\pi}{p}}z_1,e^{i\frac{2q\pi}{p}}z_2)$. Let $\tilde
f_{q,q'}:\ S^3\rightarrow S^3$, $\tilde
f_{q,q'}(z_1,z_2)=(\frac{z_1^q}{\sqrt{|z_1|^{2q}+|z_2|^{2q'}}},\frac{z_2^{q'}}{\sqrt{|z_1|^{2q}+|z_2|^{2q'}}})$.
We can check that this map induces a map ${f}_{q,q'}:\
L(p,q)\rightarrow L(p,q')$ with degree $qq'$, moreover since $q, q'$
are coprime with $p$, $f_{q,q'*}$ is an isomorphism $\pi_1$. By
Proposition \ref{6} $D_{iso}(L(p,q))=\{k^2+lp\ |\ gcd(k,p)=1\}$.
Compose each self-map on $L(p,q)$ which induces an isomorphism on
$\pi_1$ with $f_{q,q'}$, we have $\{k^2q^{-1}q'+lp\ |\ gcd(k,p)=1\}
\subset D_{iso}(L(p,q),L(p,q'))$. On the other hand, for each map
$g:\ L(p,q)\rightarrow L(p,q')$ of degree $d$ which induces an
isomorphism on $\pi_1$, then $f_{q',q }\circ g$ is a self-map on
$L(p,q)$ which induces an isomorphism on $\pi_1$, where  $f_{q',q}:\
L(p,q')\rightarrow L(p,q)$ is a degree $qq'$ map. Hence the degree
of $f_{q',q }\circ g$ is $qq'd$ which must be in $\{k^2+lp\ |\
gcd(k,p)=1\}$, that is $qq'd=k^2+lp$, $gcd(k,p)=1$, then
$d=k^2q^{-1}q'^{-1}+pl=(k/q^{-1})^2q^{-1}q'+pl\in \{k^2q^{-1}q'+lp\
|\ gcd(k,p)=1\}$. Hence $D_{iso}(L(p,q),L(p,q'))=\{k^2q^{-1}q'+lp\
|\ gcd(k,p)=1\}$.
\end{proof}

Let $U_p=\{{\rm all\ units\ in\ ring}\ \mathbb{Z}_p\}$,
$U^2_p=\{a^2\ |\ a\in U_p\}$, which is a subgroup of $U_p$. Let $H$
denote the natural projection from $\{n\in \mathbb{Z}\ |\
gcd(n,p)=1\}$ to $U_p/U_p^2$.

Later, we will omit the $p$, denote them by $U$ and $U^2$. We
consider the quotient $U/U^2=\{a_1,\cdots,a_m\}$, every $a_i$
corresponds with a coset $A_i$ of $U^2$. For the structure of $U$,
see \cite{IR} page 44, then we can get the structure of $U^2$ and
$U/U^2$ easily.

Define $\bar{A}_s=\{L(p,q_i)\ |\ q_i\in A_s\}$ (with repetition
allowed). In $U/U^2$, define $B_l=\{a_s\ |\ \# \bar{A}_s=l\}$ for
$l=1,2,\cdots$, there are only finitely many $B_l$'s are nonempty.
Let $C_l=\{a\in U/U^2\ |\ a_ia\in B_l,\ \forall a_i \in B_l\}$ if
$B_l\ne \emptyset$ and $C_l=U/U^2$ otherwise,
$C=\bigcap_{l=1}^{\infty} C_l$.

\begin{prop}
$D_{iso}(L(p,q_1)\#\cdots\#L(p,q_n))= H^{-1}(C)$.
\end{prop}

\begin{proof}
By Lemma \ref{lens}, we have
$D_{iso}(L(p,q),L(p,q'))=\{k^2q^{-1}q'+lp\ |\ gcd(k,p)=1\}$.
Therefore $D_{iso}(L(p,q),L(p,q'))$ will  not change if we replace
$L(p,q)$ by  $L(p,s^2q)$ (resp. $L(p,q')$ by $L(p,s^2q')$) for any
$s$ in $U_p$.

Now we consider the relation between two sets
$D_{iso}(L(p,q),L(p,q'))$ and $D_{iso}(L(p,q_*),L(p,q_*'))$. It is
also easy to see if $\frac{q}{q'} \frac{ q'_*}{q_*}=s^2$ in $U_p$,
then $D_{iso}(L(p,q),L(p,q'))$ $=D_{iso}(L(p,q_*),L(p,q_*))$, and if
$\frac{q}{q} \frac{q'_*}{q_*}\ne s^2$ in $U_p$, then
$D_{iso}(L(p,q),L(p,q'))$ $\bigcap
D_{iso}(L(p,q_*),L(p,q'_*))=\emptyset$.

Let $f:\ L(p,q_1)\#\cdots\#L(p,q_n)\rightarrow
L(p,q_1)\#\cdots\#L(p,q_n)$ be a map of degree $d\ne 0$. Suppose $f$
sends $L(p,q_i)$ to $L(p,q_k)$ and sends $L(p,q_j)$  to $L(p,q_l)$
in the sense of Theorem \ref{perm}. Since
$D_{iso}(L(p,q_i),L(p,q_k))$ $\bigcap D_{iso}(L(p,q_j),L(p,q_l))\ne
\emptyset$, by last paragraph, we must have $\frac{q_i}{q_k} \frac{
q_l}{q_j}=s^2$ in $U_p$. Hence $\frac{q_i}{q_j}$ is in $U^2$ if and
only if $\frac{ q_l}{q_k}$ is in $U^2$; in other words,
 $L(p,q_i)$ and $L(p,q_j)$ are in the same $\bar{A}_s$ if  and only if $L(p,q_k)$ and  $L(p,q_l)$
are in the same $\bar{A}_t$. Hence $f$  provides 1-1
self-correspondence on  $\bar{A}_1,\cdots,\bar{A}_m$, and  if some
elements in $\bar{A}_s$ corresponds to $\bar{A}_t$, there is $\#\bar
A_s=\#\bar A_t$.

Let  $f:\ L(p,q_1)\#\cdots\#L(p,q_n)\rightarrow
L(p,q_1)\#\cdots\#L(p,q_n)$ be a self-map. For each $a_i\in U/U^2$,
$f$ must send $\bar{A}_i$ to some $\bar{A}_j$ with $\#\bar
A_i=\#\bar A_j=l$, and both $a_i, a_j\in B_l$. Assume $L(p,q_i)\in
\bar{A}_i$, $L(p,q_j)\in \bar{A}_j$, then $deg(f)\in
\{k^2q_i^{-1}q_j+lp\ |\ gcd(k,p)=1\}$ by Lemma \ref{lens}. By
consider in $U/U^2$, we have $H(deg(f))=\bar q_j/\bar q_i=a_j/a_i $,
that is $H(deg(f))a_i=a_j\in B_l$. Since we choose arbitrary $a_i$
in $B_l$, we have $H(deg(f))\in C_l$. Also we choose arbitrary $l$,
we have $H(deg(f))\in \bigcap_{l=1}^{\infty} C_l=C$, hence
$deg(f)\in H^{-1}(C)$.

On the other hand, if $d\in H^{-1}(C)$, then $H(d)=c\in
C=\bigcap_{l=1}^{\infty} C_l$. For each $B_l\ne \emptyset$ and each
$a_i\in B_l$, we have $ca_i=a_j\in B_l$. Then $A_i\mapsto A_j$ gives
1-1 self-correspondence among $\{\bar A_i|\#\bar A_i=l \}$. We can
make further 1-1 correspondence from elements in $\bar A_i$ to
elements in $\bar A_j$. Since our discussion works for all $B_l\ne
\emptyset$, we have 1-1 self-correspondence on $\{L(p,q_1), ...,
L(p,q_n)\}$ (with repetition allowed). Therefore for each $L(p,
q_i)\in \bar A_i$ and $L(p, q_j)\in \bar A_j$, $c=\bar q_j \bar
q^{-1}_i$. Therefore $d$ have the form $k^2 q_j q^{-1}_i \text
{mod}\, p$ with $(k, p)=1$. By Lemma \ref{lens}, there is a map
$f_{i, j}: L(p, q_i)\to L(p, q_j)$ of degree $d$ which induces an
isomorphism on $\pi_1$.

By Lemma \ref{supset}, we can construct a self-mapping of degree $d$
of $L(p,q_1)\#\cdots\#L(p,q_n)$ which induces an isomorphism on
$\pi_1$. Hence $H^{-1}(C)\subset
D_{iso}(L(p,q_1)\#\cdots\#L(p,q_n))$. Thus
$D_{iso}(L(p,q_1)\#\cdots\#L(p,q_n))=H^{-1}(C)$.
\end{proof}

\section{$D(M$) for Nil manifolds}

\subsection{Self coverings of Euclidean orbifolds}


\begin{defn}[\cite{Sc}]
A  {\it 2-orbifold} is a Hausdorff, paracompact space which is
locally homeomorphic to the quotient space of $\mathbb{R}^2$ by a
finite group action. Suppose $\mathcal{O}_1$ and $\mathcal{O}_2$ are
orbifolds and $f: \mathcal{O}_1\rightarrow \mathcal{O}_2$ is an map.
We say $f$ is an {\it orbifold covering} if any point $p$ in
$\mathcal{O}_2$ has a neighbourhood $U$ such that $f^{-1}(U)$ is the
disjoint union of sets $V_\lambda, \lambda\in \Lambda$, such that
$f|: V_\lambda\rightarrow U$ is the natural quotient map between two
queotients of $\mathbb{R}^2$ by finite groups, one of which is a
subgroup of the other.
\end{defn}

In this paper, we only consider about orbifold with singular points.
Here we say a point $x$ in the orbifold is a {\it singular point of
index} $q$ if $x$ has a neighborhood $U$ homeomorphic to the
quotient space of $\mathbb{R}^2$ by rotate action of finite cyclic
group $\mathbb{Z}_q$, $q>1$.

An orbifold $\mathcal{O}$ with singular points $\{x_1,\cdots,x_s\}$
is homeomorphic to a surface $F$, but for the sake of the singular
points, we would like to distinguish them through denoting
$\mathcal{O}$ by $F(q_1,\cdots,q_s)$. Here $q_1,\cdots,q_s$ are
indices of singular points. Here the covering map $f:
\mathcal{O}_1\rightarrow \mathcal{O}_2$ is not the same as the
covering map from $F_1$ to $F_2$.

If $f: \mathcal{O}_1\rightarrow \mathcal{O}_2$ is an orbifold
covering, the singular points of $\mathcal{O}_2$ are
$\{x_1,\cdots,x_s\}$, for any $y\in \mathcal{O}_2, y\ne x_i$, define
$deg(f)=\#f^{-1}(y)$. For any singular point $x$, let
$f^{-1}(x)=\{a_{1},\cdots,a_{i}\}$. At point $a_j$, $f$ is locally
equivalent to $z\to z^{d_j}$ on $\mathbb{C}$, $x$ and $a_j$
correspond to $0$. Here we have$\sum d_j=d$, $a_j$ is an ordinary
point if and only if $d_j$ equals to the index of $x$. Define
$D(x)=[d_1, \cdots,d_i]$ to be the {\it orbifold covering data at
singular point $x$}, and $\mathfrak{D}(f)=\{D(x_1),\cdots,D(x_s)\}$
(with repetition allowed) to be the {\it orbifold covering data
of $f$}. 

The following lemma is easy to verify.

\begin{lem}
If a Nil manifold $M$ is not a torus bundle or a torus semi-bundle,
then $M$ has one of the following Seifert fibreing structures:
$M(0;\frac{\beta_1}{2},\frac{\beta_2}{3},\frac{\beta_3}{6})$,
$M(0;\frac{\beta_1}{3},\frac{\beta_2}{3},\frac{\beta_3}{3})$,
 or $M(0;\frac{\beta_1}{2},\frac{\beta_2}{4},\frac{\beta_3}{4})$, where $e(M) \in \mathbb{Q}-\{ 0\}$.
\end{lem}

\begin{prop}\label{prop1}
Denote the degrees set of self covering of an orbifold $\mathcal{O}$
by $D(\mathcal{O})$. We have:

(1) For $\mathcal{O}=S^2(2,3,6)$, $D(\mathcal{O})=\{m^2+mn+n^2|
m,n\in\mathbb{Z},(m,n)\ne(0,0)\}$.

Moreover, if $d\in D(\mathcal{O})$ is coprime with 6, then

(i)  $d\equiv 1 \mod 6$;

(ii) this covering map of degree $d=6k+1$ is realized by an orbifold
covering from $\mathcal{O}$ to $\mathcal{O}$ with orbifold covering
data
$$\{D(x_1),D(x_2),D(x_3)\}=\{[\underbrace{2,\cdots,2}_{3k},1],
[\underbrace{3,\cdots,3}_{2k},1],[\underbrace{6,\cdots,6}_{k},1]\},$$
where $x_1$,  $x_2$ and $x_3$ are singular points of indices 2, 3
and 6 respectively.

(2) For $\mathcal{O}=S^2(3,3,3)$, $D(\mathcal{O})=\{m^2+mn+n^2|
m,n\in\mathbb{Z},(m,n)\ne (0,0)\}$.

Moreover, if $d\in D(\mathcal{O})$ is coprime with 3, then

(i)  $d\equiv 1 \mod 3$;

(ii) this covering map of degree $d=6k+1$ is realized by an orbifold
covering from $\mathcal{O}$ to $\mathcal{O}$ with orbifold covering
data
$$\{D(x_1),D(x_2),D(x_3)\}=\{[\underbrace{3,\cdots,3}_{k},1],
[\underbrace{3,\cdots,3}_k,1],[\underbrace{3,\cdots,3}_k,1]\},$$
where $x_1$,  $x_2$ and $x_3$ are singular points of indices 3, 3
and 3 respectively.

 (3) For $\mathcal{O}=S^2(2,4,4)$,
$D(\mathcal{O})=\{m^2+n^2| m,n\in\mathbb{Z},(m,n)\ne (0,0)\}$.

Moreover, if $d\in D(\mathcal{O})$ is coprime with 4, then

(i)  $d\equiv 1 \mod 4$;

(ii) this covering map of degree $d=4k+1$ is realized by an orbifold
covering from $\mathcal{O}$ to $\mathcal{O}$ with orbifold covering
data
$$\{D(x_1),D(x_2),D(x_3)\}=\{[\underbrace{2,\cdots,2}_{2k},1],[\underbrace{4,\cdots,4}_k,1],
[\underbrace{4,\cdots,4}_k,1]\},$$ where $x_1$,  $x_2$ and $x_3$ are
singular points of indices 2, 4 and 4 respectively.

\end{prop}
\begin{proof}
We only prove case (1). The other two cases can be proved similarly.

$S^2(2,3,6)$ can be seen as pasting the equilateral triangle as
shown in figure 1 geometrically.

\begin{center}
\psfrag{a}[]{a} \psfrag{b}[]{b}
\includegraphics[height=2in]{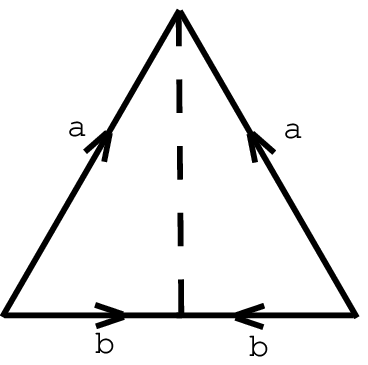}
\centerline{Figure 1}
\end{center}

\begin{center}
\psfrag{a}[]{a} \psfrag{b}[]{b} \psfrag{c}[]{0}
\psfrag{d}[]{1}\psfrag{e}[]{$e^{i\frac{\pi}{3}}$}
\includegraphics[height=3in]{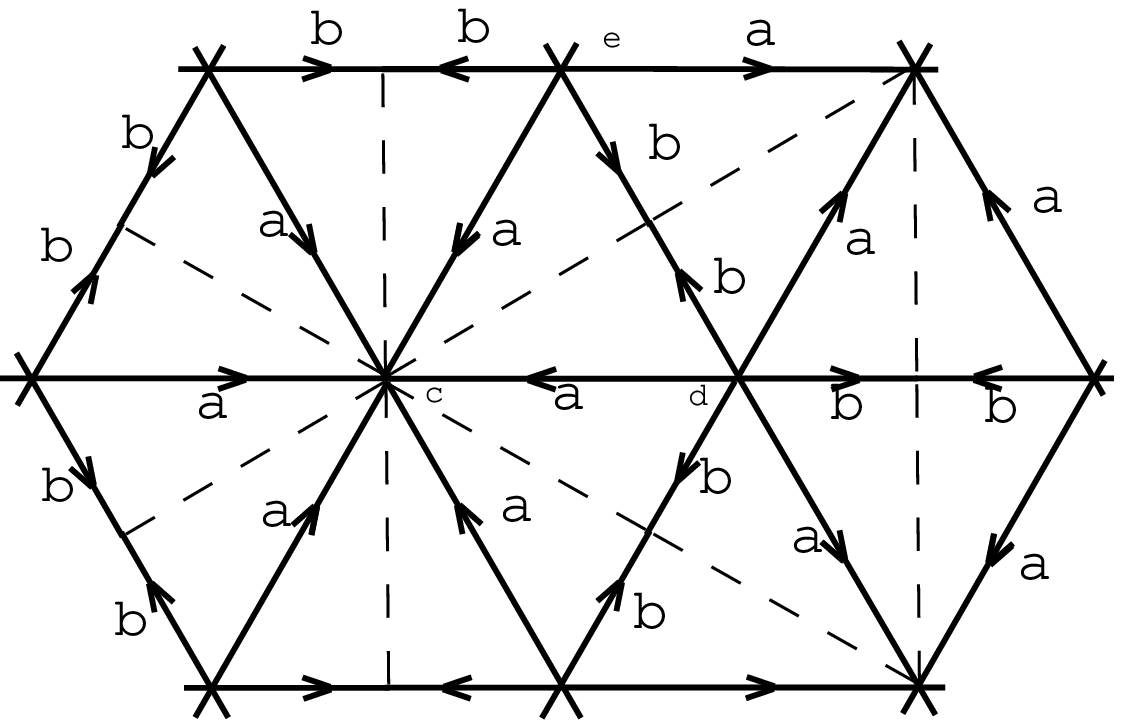}
\centerline{Figure 2}
\end{center}

$\pi_1(S^2(2,3,6))$ can be identified with a discrete subgroup
$\Gamma$ of $Iso_+(\mathbb{E}^2)$, a fundamental domain of $\Gamma$
is shown in figure 2. It is as a lattice in $\mathbb{E}^2$ with
vertex coordinate $m+ne^{i\frac{\pi}{3}},\ m,n\in\mathbb{Z}$.

For the covering $p:\ T^2\rightarrow S^2(2,3,6)$, $T^2$ can be seen
as the quotient of a subgroup $\Gamma' \subset \Gamma$ on
$\mathbb{E}^2$, with a fundamental domain as figure 3. Here
$\Gamma'$ is just all the translation elements of $\Gamma$, thus
$\Gamma'$ is generated by $z\rightarrow z+\sqrt{3}i$ and
$z\rightarrow z+\frac{\sqrt{3}}{2}i+\frac{3}{2}$.

\begin{center}
\psfrag{a}[]{0} \psfrag{b}[]{1}\psfrag{c}[]{$e^{i\frac{\pi}{3}}$}
\includegraphics[height=3.7in]{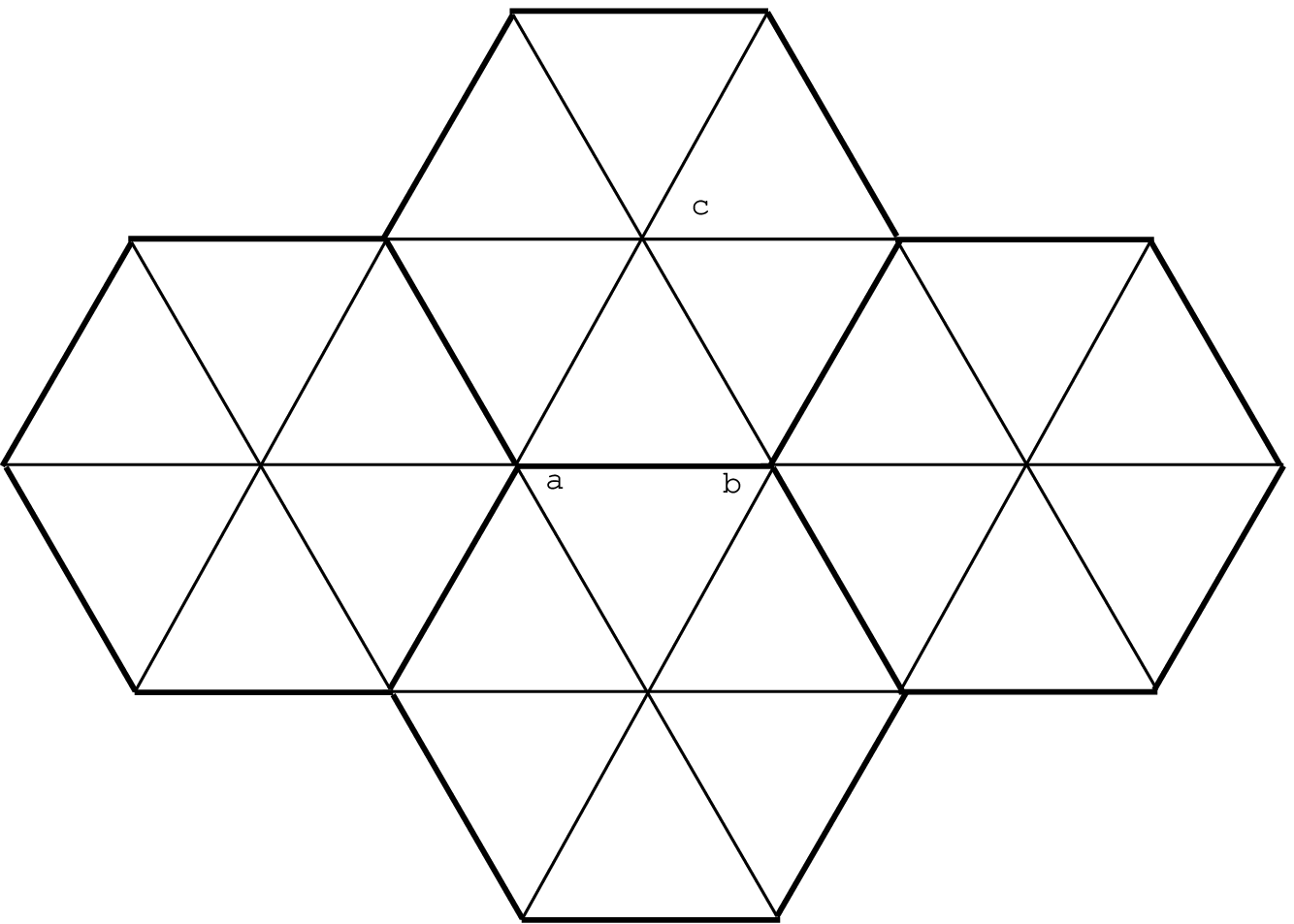}
\vskip 0.5 truecm \centerline{Figure 3}
\end{center}

For every self covering $f:\ S^2(2,3,6) \rightarrow S^2(2,3,6)$,
$f_*:\ \pi_1(S^2(2,3,6))\rightarrow\pi_1(S^2(2,3,6))$ is injective.
Since $p$ is covering, $f_*\circ p_*:\ \pi_1(T^2)\rightarrow
\pi_1(S^2(2,3,6))$ is also injective. So $f_*(p_*(\pi_1(T^2)))$ is a
free abelian subgroup of $\pi_1(S^2(2,3,6))$.

For every $\gamma\in\Gamma$, which is not translation, it can be
represented by $f:\ z\rightarrow e^{i\frac{2k\pi}{n}}z+z_0,\
gcd(k,n)=1,\ n>1$. Then $f^n(z)=(e^{i\frac{2k\pi}{n}})^n z+
(e^{i\frac{2k(n-1)\pi}{n}}+\cdots e^{i\frac{2k\pi}{n}}+1)z_0=z$. So
$\gamma$ is a torsion element, thus $\gamma \notin
f_*(p_*(\pi_1(T^2)))$ except $\gamma=e$. So $f_*(p_*(\pi_1(T^2)))
\subset p_*(\pi_1(T^2))$, thus there exists $\tilde{f}:\
T^2\rightarrow T^2$ being the lifting of $f$.
\[ \begin{CD}
T^2@> \tilde{f} >>
T^2\\
@V p VV                              @VV pV\\
S^2(2,3,6) @> f >> S^2(2,3,6).
\end{CD}  \]

Here we have
$$deg(f)=deg(\tilde{f})=[\pi_1(T^2):\tilde{f}_*(\pi_1(T^2))]=
\frac{{\rm area(fundamental\ domain\ of}\
\tilde{f}_*(\pi_1(T^2)))}{{\rm area(fundamental\ domain\ of} \
\pi_1(T^2))},$$ here $\tilde{f}_*(\pi_1(T^2)), \pi_1(T^2)$ are all
seen as subgroup of $\pi_1(S^2(2,3,6))$.

Clearly, we can choose a fundamental domain of
$f_*(\pi_1(S^2(2,3,6)))$ to be an equilateral triangle in
$\mathbb{E}^2$ with vertices as $m+ne^{i\frac{\pi}{3}}$,  then the
fundamental domain of $\tilde{f}_*(\pi_1(T^2))$ is an equilateral
hexagon with vertices as $m+ne^{i\frac{\pi}{3}}$. The scale of area
is the square of the scale of edge length. The scale of edge length
must be $|m+ne^{i\frac{\pi}{3}}|=\sqrt{m^2+mn+n^2}$. So
$deg(f)=m^2+mn+n^2$.

On the other hand, for every $(m,n)\in\mathbb{Z}^2-\{(0,0)\}$,
choose $g:\ \mathbb{E}^2\rightarrow \mathbb{E}^2$,
$g(z)=(m+ne^{i\frac{\pi}{3}})z$. It is routine to check that for any
$\gamma \in \Gamma$, there is $\gamma' \in \Gamma$, such that
$g(\gamma(z))=\gamma'(g(z))$. So $g$ induces $\bar{g}$, which is
self covering on $S^2(2,3,6)$, and $deg(\bar{g})=m^2+mn+n^2$. We
have proved the first sentence of Proposition \ref{prop1} (1).

If $m^2+mn+n^2$ is coprime to 6, $m^2+mn+n^2\equiv 1\ {\rm or}\ 5\
\text{mod}\ 6$. Since $m^2+mn+n^2\equiv 4m^2+4mn+4n^2 \equiv
(2m+n)^2 \ \text{mod}\ 3$, and any square number must be $0$ or $1$
{mod} $3$, we must have $m^2+mn+n^2\equiv1\ \text{mod}\ 6$. We have
proved Proposition \ref{prop1} (1) (i).

Assume $h$ is a self covering of degree $d=6k+1$, $x_1,x_2,x_3$ are
the singular points on $S^2(2,3,6)$ with indices $2,3,6$. For $x_1$,
$h^{-1}(x_1)$ must be ordinary points or singular point of index
$2$. Since the degree $d=6k+1$, $h^{-1}(x_1)$ is $3k$ ordinary
points and $x_1$. Similarly, for $x_2$, $h^{-1}(x_2)$ is $2k$
ordinary points and $x_2$. Then $x_1,x_2\notin h^{-1}(x_3)$, so
$h^{-1}(x_3)$ is $k$ ordinary points and $x_3$. Thus the covering
map of degree $d=6k+1$ is realized by a self covering of
$\mathcal{O}$ with orbifold covering data
$\{[2,\cdots,2,1],[3,\cdots,3,1],[6,\cdots,6,1]\}$. We have proved
Proposition \ref{prop1} (1) (ii).
\end{proof}

\subsection{$D(M)$ for Nil manifolds}

\begin{thm} For
3-manifold $M$ in Class 4, we have

(1) For
$M=M(0;\frac{\beta_1}{2},\frac{\beta_2}{3},\frac{\beta_3}{6})$,
$D(M)=\{l^2|l=m^2+mn+n^2, l\equiv 1 \mod 6, m,n \in \mathbb{Z}\}$;

(2) For
$M=M(0;\frac{\beta_1}{3},\frac{\beta_2}{3},\frac{\beta_3}{3})$,
$D(M)=\{l^2|l=m^2+mn+n^2, l\equiv 1 \mod 3, m,n \in \mathbb{Z}\}$;

(3) For
$M=M(0;\frac{\beta_1}{2},\frac{\beta_2}{4},\frac{\beta_3}{4})$,
$D(M)=\{l^2|l=m^2+n^2, l\equiv 1 \mod 4, m,n \in \mathbb{Z}\}$.
\end{thm}

\begin{proof} We will just prove Case (1). The proof of Cases (2)
and (3) are exactly as that of Case (1). Below
$M=M(0;\frac{\beta_1}{2},\frac{\beta_2}{3},\frac{\beta_3}{6})$.

First we show that $D(M)\subset \{l^2|l=m^2+mn+n^2, l\equiv 1 \mod
6, m,n \in \mathbb{Z}\}$.

Suppose $f$ is a self map of $M$. By \cite[Corollary 0.4]{Wa}, $f$
is homotopic to a covering map $g: M\rightarrow M$. Since $M$ has
unique Seifert fibering structure up to isomorphism, we can make $g$
to be a fiber preserving map. Denote the orbifold of $M$ by $O_M$.
By \cite[Lemma 3.5]{Sc}, we have:
\begin{equation}\label{eq1}
\left\{ \begin{array}{l}
         e(M)=e(M)\cdot \frac{l}{m}, \\
         deg(g)=l\cdot m,
         \end{array} \right.
\end{equation}
where $l$ is the covering degree of $O_M\rightarrow O_M$ and $m$ is
the fiber degree.

Since $e(M)\neq 0$, from equation (\ref{eq1}) we get $l=m$. Thus
$deg(f)=deg(g)$ is a square number $l^2$. Since the orbifold
$O_M=S^2(2,3,6)$, by Proposition \ref{prop1} $(1)$, we have
$l=m^2+mn+n^2$. Below we show that $l=6k+1$.

Let $N$ be the regular neighborhood of 3 singular fibers. To define
the Seifert invariants,  a section $F$ of $M\setminus N$ is chosen,
and moreover  $\partial F$ and fibers on each component of $\partial
(M\setminus N)$ are oriented.

Consider the covering $g|: M\setminus g^{-1}(N)\to M\setminus N$.
Let $\tilde F$ be a component $g^{-1}(F)$. It is easy to verify that
$\tilde F$ is a section of $M\setminus g^{-1}(N)$. Now we lift the
orientations on $\partial F$ and the fibers on $\partial (M\setminus
N)$ to those on $\partial (M\setminus g^{-1}(N))$, we get a
coordinate system on $\partial (M\setminus g^{-1}(N))$. Therefore we
have a coordinate preserving covering

$$g: (M, M\setminus g^{-1}(N), g^{-1}(N))\to (M, M\setminus N, N).$$

Suppose $V'$ is a tubular neighborhood of some singular fiber $L'$.
The meridian of $V'$ can be represented by $(c')^{\alpha
'}(h')^{\beta '}$ ($\alpha '>0$), where $(c',h')$ is the section-
fiber coordinate of $\partial V'$.

\begin{center}
\psfrag{a}[]{meridian of $V'$} \psfrag{b}[]{meridian of
$V$}\psfrag{c}[]{$c'=$ a component of $\partial F$}\psfrag{d}[]{$c=$
a component of $\partial \tilde{F}$}
\psfrag{e}[]{$F$}\psfrag{f}[]{$\tilde{F}$}\psfrag{g}[]{$\partial
V'$}\psfrag{h}[]{$\partial V$}\psfrag{i}[]{$h'=$
 fiber}\psfrag{j}[]{$h=$ fiber}\psfrag{k}[]{$f$}
\includegraphics{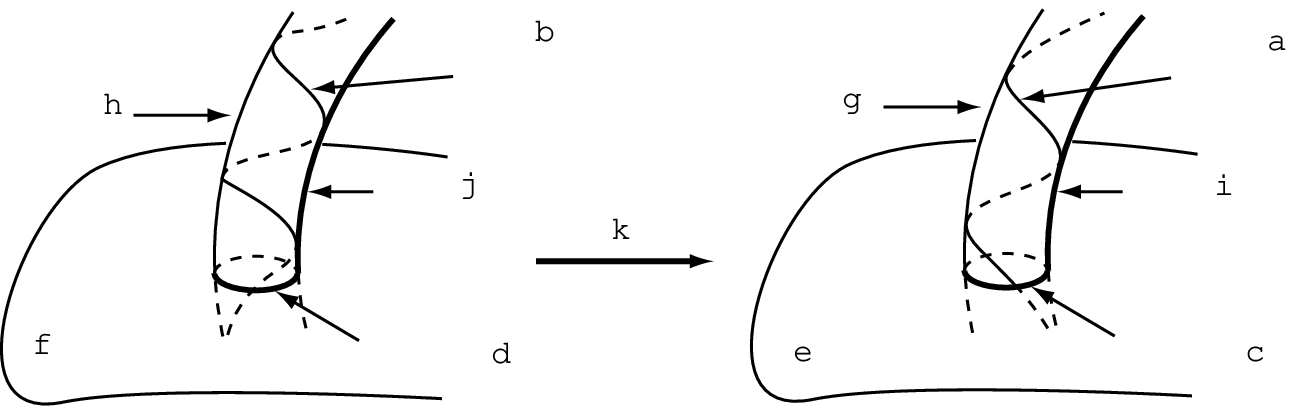}
\vskip 0.5 truecm \centerline{Figure 4}
\end{center}

Suppose $V$ is a component  of $g^{-1}(V')$  and the meridian of $V$
is represented as $c^{\alpha}h^{\beta}$ ($\alpha>0$), where $(c, h)$
is the lift of $(c', h')$. Since $g|:\ V\rightarrow V'$ is a
covering of solid torus, so $g$ must send meridian to meridian
homeomorphically, thus $g(c^{\alpha}h^{\beta})=(c')^{\alpha
'}(h')^{\beta '}$. See Figure 4.

Since  $g$ has the fiber degree $m=l$, $g(h)=(h')^l$. Since $c,c'$
are the boundaries of sections and $g$ send $c$ to $c'$, we have
$g(c)=(c')^s$. Then $ g(c^{\alpha}h^{\beta})=(c')^{\alpha\cdot
s}(h')^{\beta\cdot l}=(c')^{\alpha '}(h')^{\beta '} $. Hence  we get
$\beta\cdot l=\beta'$.

Let $V'$ be a tubular neighborhood of singular fiber whose meridian
can be represented as $(c')^6(h')^{\beta '}$. By the arguments
above, the meridian of the preimage $V$ can be represent by
$c^\alpha h^\beta$.

Since $\beta '$ is coprime with 6. By $\beta\cdot l=\beta'$, so $l$
is coprime with 6. Still by Proposition \ref{prop1} $(1)$, we have
$l=6k+1$.

Then we show $\{l^2|l=m^2+mn+n^2, l\equiv 1 \mod 6, m,n \in
\mathbb{Z}\}\subset D(M)$.

Suppose $l=m^2+mn+n^2$ and $l=6k+1$, denote the quotient manifold of
$\mathbb{Z}_l$ free action on $M$ by $M_l$. Then $M_l$ has the
Seifert fibering structure
$M(0;l\cdot\frac{\beta_1}{2},l\cdot\frac{\beta_2}{3},l\cdot\frac{\beta_3}{6})$.
We have the covering $g_l: M\to M_l$ of degree $l$.

\textbf{Claim:} there exists a map $f_l: M_l\rightarrow M$ of degree
$l$.

Let $D=D_1 \bigcup D_2 \bigcup D_3\subset S^2(2,3,6)$ be the regular
neighborhood discs of 3 singular points of indices 2, 3, and 6
respectively. By Proposition \ref{prop1} (1), there exists a
branched covering map $\bar {f_l}: S^2(2,3,6)\rightarrow S^2(2,3,6)$
of degree $l$ such that

(1) $\bar f_l$ induce a covering map $\bar f_l|:  S^2\setminus \bar
{f_l}^{-1}(D)\rightarrow S^2\setminus D$;

(2) $\bar {f_l}^{-1}(D_i)$ consists of  $(3k+1)$ discs with orbifold
covering data $[\underbrace{2,\cdots,2}_{3k},1]$ for $i=1$, and
$(2k+1)$ discs with orbifold covering data
$[\underbrace{3,\cdots,3}_{2k},1]$ for $i=2$, and $(k+1)$ discs with
orbifold covering  data
 $[\underbrace{6,\cdots,6}_{k},1]$ for $i=3$.

Clearly $\bar f_l^{-1}(D)$ consists of $(3k+1)+(2k+1)+(k+1)=6k+3$
disks.

Then we have the covering map $\bar{f_l}\times  id:
(S^2\setminus{f_l}^{-1}(D))\times S^1 \rightarrow (S^2\setminus
 D)\times S^1$ of degree $l$, which can be extends to a covering map
 $f_l: M'\rightarrow M$,  where $M'$ has the Seifert structure
$M(0;\underbrace{\beta_1,\cdots,\beta_1}_{3k},\frac{\beta_1}{2}
,\underbrace{\beta_2,\cdots,\beta_2}_{2k},\frac{\beta_2}{3},\underbrace{\beta_3,\cdots,\beta_3}_{k},\frac{\beta_3}{6})$.
Clearly $M'$ is isomorphic to $M_l$.

Now the covering $f_l\circ g_l: M\rightarrow M_l \rightarrow M$ has
degree $l^2$.

We finish the proof of Case (1).\end{proof}

\section{$D(M)$ for $\mathbb{H}^2\times \mathbb{E}^1$ manifolds}

In this case, all the manifolds are Seifert manifolds $M$ such that
the Euler number $e(M)=0$ and the Euler characteristic  of the
orbifold $\chi(O_M)<0$.

Suppose
$M=(g;\frac{\beta_{1,1}}{\alpha_1},\cdots,\frac{\beta_{1,m_1}}{\alpha_1},\cdots,\frac{\beta_{n,1}}{\alpha_n},\cdots,\frac{\beta_{n,m_n}}{\alpha_n})$,
where all the integers $\alpha_i>1$ are different from each other,
and $\sum_{i=1}^n \sum_{j=1}^{m_n} \frac{\beta_{i,j}}{\alpha_i}=0$.

For every $\alpha_i$, consider $U_{\alpha_i}$. For every $a\in
U_{\alpha_i}$, define $\theta_a(\alpha_i)=\#\{\beta_{i,j}\ |\
p_i(\beta_{i,j})=a\}$ (with repetition allowed), where $p_i$ is the
natural projection from $\{n\ |\ gcd(n,\alpha_i)=1\}$ to
$U_{\alpha_i}$. Define $B_l(\alpha_i)=\{a\ |\
\theta_a(\alpha_i)=l\}$ for $l=0,1,\cdots$, there are only finitely
many $B_l(\alpha_i)$ nonempty. Let $C_l(\alpha_i)=\{b\in
U_{\alpha_i}\ |\ ab\in B_l(\alpha_i),\ \forall a \in
B_l(\alpha_i)\}$ if $B_l(\alpha_i)\ne \emptyset$ and
$C_l(\alpha_i)=U_{\alpha_i}$ otherwise. Finally define
$C(\alpha_i)=\bigcap_{l=1}^{\infty} C_l(\alpha_i)$, and
$\bar{C}(\alpha_i)=p_i^{-1}(C(\alpha_i))$.

\begin{thm} $$D(M(g;\frac{\beta_{1,1}}{\alpha_1},\cdots,\frac{\beta_{1,m_1}}{\alpha_1},\cdots,\frac{\beta_{n,1}}{\alpha_n},\cdots,\frac{\beta_{n,m_n}}{\alpha_n}))=\bigcap_{i=1}^n
\bar{C}(\alpha_i).$$
\end{thm}

\begin{proof} Suppose $f$ is a non-zero degree self-mapping of $M$. By \cite[Corollary 0.4]{Wa}, $f$
is homotopic to a covering map $g: M\rightarrow M$. Since $M$ has
the unique Seifert structure, we can isotopy $g$ to a fiber
preserving map. Denote the orbifold of $M$ by $O_M$. Then $g$
induces a self-covering $\bar{g}$ on $O_M$, since $\chi(O_M)<0$,
then $\bar{g}$ must be $1$-sheet, thus isomorphism of $O_M$.

So $g$ is a degree $d$ covering on the fiber direction. Or
equivalently, by the action of $\mathbb{Z}_d$ on each fiber, the
quotient of $M$ is also $M$. Thus $d \in D(M)$ if and only if
$$M'=M(g;d\frac{\beta_{1,1}}{\alpha_1},\cdots,d\frac{\beta_{1,m_1}}{\alpha_1},\cdots,d\frac{\beta_{n,1}}{\alpha_n},\cdots,d\frac{\beta_{n,m_n}}{\alpha_n})$$
is homeomorphic to $M$.

By the uniqueness of Seifert structure (\cite{Sc} Theorem 3.9) and
the fact $e(M)=0$, we have that $M$ is homeomorphism to $M'$ if and
only if
$(\beta_{i,1},\cdots,\beta_{i,m_i})=(d\beta_{i,1},\cdots,d\beta_{i,m_i})$
under a permutation, all the numbers are seen as in $U(\alpha_i)$.

For every $a\in U(\alpha_i)$, if
$(\beta_{i,1},\cdots,\beta_{i,m_i})=(d\beta_{i,1},\cdots,d\beta_{i,m_i})$
holds, we must have $\theta_a(\alpha_i)=\theta_{da}(\alpha_i)$, thus
$p_i(d)\in C_{\theta_a}(\alpha_i)$. For $a$ is an arbitrary element
in $U(\alpha_i)$, we have $p_i(d)\in C(\alpha_i)$, thus $d\in
\bar{C}(\alpha_i)$. Since $\alpha_i$ is also chosen arbitrarily,
$d\in \bigcap_{i=1}^n \bar{C}(\alpha_i)$, thus $D(M)\subset
\bigcap_{i=1}^n \bar{C}(\alpha_i)$.

For any $d\in \bigcap_{i=1}^n \bar{C}(\alpha_i)$, $M$ is
homeomorphic to $M'$, so $D(M)\supset \bigcap_{i=1}^n
\bar{C}(\alpha_i)$
\end{proof}

\textbf{Acknowledgement.} The authors are partially supported by
grant No.10631060 of the National Natural Science Foundation of
China and Ph.D. grant No. 5171042-055 of the Ministry of Education
of China.


\bibliographystyle{amsalpha}

 Hongbin Sun  \hskip 1true cm e-mail: hongbin.sun2331@gmail.com

School of Mathematical Sciences, Peking University, Beijing 100871,
China

\vskip 0.3 true cm

Shicheng Wang \hskip 1true cm e-mail: wangsc@math.pku.edu.cn

School of Mathematical Sciences, Peking University, Beijing 100871,
China

\vskip 0.3 true cm

Jianchun Wu \hskip 1true cm  e-mail: wujianchun@math.pku.edu.cn

School of Mathematical Sciences, Peking University, Beijing 100871,
China

\vskip 0.3 true cm

Hao Zheng \hskip 1true cm e-mail: zhenghao@mail.sysu.edu.cn

Depart. of Math., Zhongshan University, Guangzhou 510275, China

\end{document}